\documentclass[11pt,reqno]{amsart}
\usepackage{mathrsfs}
\usepackage{graphicx}
\usepackage{amssymb}
\usepackage{amsmath}
\usepackage{amsthm}
\usepackage{amscd}
\usepackage{epsfig}

\pdfoutput=1
\usepackage{amsfonts}
\usepackage{amsthm}
\usepackage{amscd}
\usepackage{amsmath}
\usepackage{txfonts}
\usepackage{mathrsfs}
\usepackage[colorlinks, linkcolor=blue]{hyperref}

\theoremstyle{definition}
\newtheorem{defn}{Definition}[section]

\theoremstyle{plain}
\newtheorem{thm}{Theorem}[section]
\newtheorem{prop}[thm]{Proposition}
\newtheorem{lem}[thm]{Lemma}

\newtheorem{cor}[thm]{Corollary}
\newtheorem{exa}[thm]{Example}

\newtheorem*{theorem}{Theorem}

\newtheorem*{clm}{Claim}

\theoremstyle{remark}
\newtheorem*{rmk}{Remark}

\newcommand{\talpha}{\tilde{\alpha}}
\newcommand{\hgamma}{\widehat{\gamma}}
\newcommand{\hxi}{\widehat{\xi}}

\renewcommand{\le}{\leqslant}
\renewcommand{\ge}{\geqslant}

\newcommand{\mbbe}{\mathbb{E}}

\newcommand{\mbbh}{\mathbb{H}}

\newcommand{\mbbn}{\mathbb{N}}

\newcommand{\mbbr}{\mathbb{R}}

\newcommand{\mbbz}{\mathbb{Z}}

\newcommand{\mclb}{\mathcal{B}}
\newcommand{\mclc}{\mathcal{C}}

\newcommand{\mclf}{\mathcal{F}}
\newcommand{\mclg}{\mathcal{G}}

\newcommand{\mcll}{\mathcal{L}}
\newcommand{\mclm}{\mathcal{M}}

\newcommand{\mclv}{\mathcal{V}}

\newcommand{\mscb}{\mathscr{B}}
\newcommand{\mscc}{\mathscr{C}}

\newcommand{\msce}{\mathscr{E}}
\newcommand{\mscf}{\mathscr{F}}

\newcommand{\dif}{\mathrm{d}}
\newcommand{\Dif}{\mathrm{D}}

\newcommand{\supp}{\mathrm{supp}}

\DeclareMathOperator{\diam}{diam} 

\newcommand{\qtext}[1]{\quad\text{#1}\quad}

\newcommand{\textand}{\quad\text{and}\quad}

\numberwithin{equation}{section}

\begin{document}

\title[Local stable and unstable sets for positive entropy $C^1$ dynamical systems]
 {Local stable and unstable sets for positive entropy $C^1$ dynamical systems} 

\author[S. Feng]{Shilin Feng}
\address{S. Feng: College of Mathematics, Sichuan University, Chengdu, China, 610064}
\email{fengshilinscu@163.com}

\author[R. Gao]{Rui Gao}
\address{R. Gao: College of Mathematics, Sichuan University, Chengdu, China, 610064}
\email{gaoruimath@scu.edu.cn}

\author[W. Huang]{Wen Huang}
\address{W. Huang: Wu Wen-Tsun Key Laboratory of Mathematics, USTC, Chinese Academy of Sciences and Department of Mathematics, University of Science and
Technology of China, Hefei, Anhui 230026, China}
\email{wenh@mail.ustc.edu.cn}

\author[Z. Lian]{Zeng Lian}
\address{Z. Lian: College of Mathematics, Sichuan University, Chengdu, China, 610064}
\email{ZengLian@gmail.com}

\subjclass[2010]{37A35, 37C45, 37D25}
\thanks{Gao is partially supported by NNSF of China (11701394). Huang is partially supported by NNSF of China (11431012,11731003). Lian is partially supported
by NNSF of China (11671279, 11541003). }


\maketitle

\begin{abstract}
For any $C^1$ diffeomorphism on a smooth compact Riemannian manifold that admits an ergodic measure with positive entropy,  a lower bound of the Hausdorff dimension for the local stable and unstable sets is given in terms of the measure-theoretic entropy and the maximal Lyapunov exponent. The mainline of
 our approach to this result is under the settings of topological dynamical systems, which is also applicable to infinite dimensional $C^1$ dynamical systems.
\end{abstract}


\section{Introduction}\label{intro}

Let $M$ be a compact manifold and let $f:M\to M$ be a $C^{1+\alpha}$ diffeomorphism for some $\alpha>0$. The stable manifold theory developed by Pesin among others \cite{FHY,Ka,Pe76,Pe77,Ru79} asserts that, roughly speaking, if $f$ is {\it non-uniformly hyperbolic} with respect to an $f$-invariant Borel probability measure $\mu$, then the stable and unstable sets for $\mu$-a.e. $x\in M$ are immersed submanifolds with complementary dimensions. To indicate that the $C^{1+\alpha}$  regularity hypothesis in Pesin's stable manifold theory is essential, Pugh \cite{Pu} gave an example of a $C^1$-diffeomorphism which admits an orbit with nonzero Lyapunov exponents but no invariant manifolds.

Pugh's counter-example is about the non-existence of invariant manifolds for a single orbit of a concrete $C^1$ diffeomorphism. More recently, Bonatti, Crovisier and Shinohara \cite{BCS} showed that non-existence of invariant manifolds is a generic phenomenon in the $C^1$ category. More precisely, for a diffeomorphism $f:M\to M$  on some Riemannian manifold $M$, let $d$ be the metric on $M$ induced by its Riemannian structure. Then for $x\in M$, the
{\it stable set} of $x$ for $f$ is defined by
$$W^s(x,f) := \{y \in M : d(f^n(x), f^n(y))\to 0 \text{ as } n \to +\infty\}.$$
The {\it unstable set} $W^u(x,f)$ is simply defined as the stable set of $x$ for $f^{-1}$. In \cite[Theorem~2]{BCS} the authors proved the following statement.

\begin{theorem}
  Let $M$ be a smooth compact manifold with $\dim M\ge 3$ and let  ${\text{Diff}}^1(M)$ be endowed with the $C^1$-topology. Then there exists a non-empty open set $\mathcal{U}\subset {\text{Diff}}^1(M)$ and a dense $G_\delta$ subset $\mathcal{R}\subset \mathcal{U}$ satisfying the following property: each $f\in \mathcal{R}$ admits a hyperbolic and ergodic Borel probability measure $\mu$ such that
$$W^s(x,f)=W^u(x,f)=\{x\}$$
for every point $x$ in the support $\supp(\mu)$ of $\mu$.
\end{theorem}

According to \cite[Remark~3]{BCS}, for each $f\in \mathcal{R}$ in the above Theorem, the dynamics of $f$ on  $\supp(\mu)$ is a generalized adding machine (also called odometer or solenoid, see for example \cite{BS} or \cite{MM} for the definition); therefore, it is uniquely ergodic and the measure-theoretic entropy $h_\mu(f)=0$. Then a natural question arises: given $f\in {\text{Diff}}^1(M)$ and $f$-invariant $\mu$ with $h_\mu(f)>0$, what can we say about the structure of $W^s(x,f)$ and $W^u(x,f)$ for $\mu$-a.e. $x\in M$?

As a partial answer of the question above, the purpose of this paper is to investigate the Hausdorff dimension of the local stable sets and local unstable sets for  $C^1$ dynamical systems with positive measure-theoretic entropy (see also problem 6.2 in \cite{FGH}).

Throughout  this  paper, by a {\it topological dynamical system} $(X,T)$ (TDS for short) we mean a compact metric space $(X,d)$ with a homeomorphism map $T$ from $X$ onto itself, where $d$ refers to the metric on $X$. For a TDS $(X,T)$, given $x\in X$ and $\delta>0$, the {\it $\delta$-stable set} of $x$ is defined as
\begin{equation}\label{eq:local stable set}
  W_\delta^s(x,T)=\big\{y\in X:  d(T^n x,T^n y)\le \delta, ~\forall n\ge 0 \ \text{ and }\lim\limits_{n\rightarrow +\infty} d(T^n x,T^n y)=0\big\}\,.
\end{equation}
Similarly,  the {\it  $\delta$-unstable set} of $x$ is defined as
\begin{equation}\label{eq:local unstable set}
 W_\delta^u(x,T)=\big\{y\in X:  d(T^{-n}x,T^{-n}y)\le \delta, ~\forall n\ge 0 \ \text{ and } \lim\limits_{n\rightarrow +\infty} d(T^{-n}x, T^{-n}y)=0 \big\}\,.
\end{equation}
Clearly, by definition, $W_\delta^u(x,T)=W_\delta^s(x,T^{-1})$. $\delta$-stable(resp. unstable) sets for some $\delta>0$ are collectively called {\it local} stable(resp. unstable) sets.

Let $M$ be a smooth compact Riemannian manifold. The Riemanian structure on $M$ induces a complete metric $d$ in the usual way. The Hausdorff dimension for subsets of $M$ defined by this metric is denoted by $\dim_H(\cdot)$. Given a $C^1$ diffeomorphism $f$ on $M$,  let $\Dif f^n(x)$ denote the tangent map of $f^n$ at $x$ and  $\|\Dif f^n(x)\|$ the operator norm of $\Dif f^n(x)$ induced by the Riemannian metric.  For an $f$-invariant Borel probability measure  $\mu$ on $M$, let us define the {\it maximal Lyapunov exponent} of $f$ w.r.t. $\mu$ by
\begin{equation}\label{eq:Lyapunov exp C1}
  \chi_\mu(f) :=\lim_{n\to +\infty} \frac{1}{n}\int_M \log^+\|\Dif f^n(x)\| \dif \mu(x)\,,
\end{equation}
where $\log^+ t=\max\{\log t,0\}$ for $t\ge 0$. Let $h_\mu(f)$ denote the measure-theoretic entropy of the measure preserving system $(M,f,\mu)$. The main result in this paper is as follows:

\begin{thm}\label{thm:main C1 finite}
Let $M$ be a smooth compact Riemannian manifold and $f:M\to M$ a $C^1$ diffeomorphism. Let $\mu$ be an ergodic $f$-invariant Borel probability measure on $M$ with positive entropy. Then the following holds for every $\delta>0$:
\begin{equation}\label{eq:main C1 finite}
  \dim_H\big( W_\delta^u(x,f)\big)\ge\frac{h_\mu(f)}{\chi_\mu(f)} \textand \dim_H \big( W_\delta^s(x,f)\big)\ge\frac{h_\mu(f)}{\chi_\mu(f^{-1})}  \,,\quad \mu\text{-a.e.}~x\in M \,.
\end{equation}
\end{thm}

\begin{rmk}\mbox{}
\begin{enumerate}
  \item According to Margulis-Ruelle inequality, the assumption $h_\mu(f)>0$ guarantees that both $\chi_\mu(f)$ and $\chi_\mu(f^{-1})$ are strictly positive.
  \item Without further assumption on the system $(M,f,\mu)$, the lower bound estimate of Hausdorff dimensions in \eqref{eq:main C1 finite} cannot be improved, even if $f$ has distinct positive(or negative) Lyapunov exponents. For example, consider the product system $(M,f,\mu)$ of two systems $(M_i,f_i,\mu_i)$, $i=1,2$ as follows. $(M_1,f_1,\mu_1)$ satisfies that $\mu_1$ is an ergodic hyperbolic measure for $(M_1,f_1)$ such that $$W^u(p,f_1)=W^s(p,f_1)=\{p\}$$ for each $p\in\supp(\mu_1)$, whose existence is guaranteed by Theorem 2 in \cite{BCS} we cited before; in particular, $h_{\mu_1}(f_1)=0$. $(M_2,f_2)$ is a two-dimensional hyperbolic torus automorphism induced by some matrix $A\in \mathrm{SL}(2,\mbbz)$
  and $\mu_2$ is the Lebesgue measure on $M_2$. Let $\lambda>\lambda^{-1}$ be the two eigenvalues of $A$ and we further require that $$\lambda>\max_{p\in N}\max\{\|\Dif f_1(p)\|,\|\Dif f_1^{-1}(p)\|\},$$ say
  $A=\begin{pmatrix}
    n & n+1 \\
    n-1 & n
  \end{pmatrix}$ for large $n$. Then $$\chi_{\mu}(f^{\pm 1})=\chi_{\mu_2}(f_2^{\pm 1})=h_\mu(f)=h_{\mu_2}(f_2)=\log\lambda,$$
   and $\dim_H\big( W^u(x,f)\big)=\dim_H \big( W^s(x,f) \big)=1$ for any $x\in\supp (\mu_1)\times M_2$.
\end{enumerate}

\end{rmk}

  In fact, we shall prove a more general statement for TDS, which contains Theorem~\ref{thm:main C1 finite} as a special case. Let $(X,T)$ be a TDS with metric $d$ as before. Given $n\ge 1$ and $x\in X$, the {\it pointwise Lipschitz constant} of $T^n$ at $x$ is defined as follows:
\begin{equation}\label{eq:Lip const point}
  \mcll_n(x):=\lim_{r\searrow 0}  \left(\sup_{y\in B(x,r)\setminus\{x\}}\frac{d(T^n x,T^n y)}{d(x,y)} \right) \in [0,+\infty] \,,
\end{equation}
where $B(x,r):=\{y\in X: d(x,y)<r\}$ denotes the open ball of radius $r$ centered at $x$, and we adopt the convention that $\sup\varnothing=0$ in \eqref{eq:Lip const point}.  Then we have:

\begin{thm}\label{thm:main TDS}
Let $(X,T)$ be a TDS and let $\mu$ be an ergodic $T$-invariant Borel probability measure on $X$ with positive entropy. Suppose that there exists $r>0$ such that
\begin{equation}\label{eq:Lip cond}
  \int_X \log^+\left(\sup_{y\in B(x,r)\setminus\{x\}}\frac{d(Tx,Ty)}{d(x,y)}\right)\dif \mu <\infty \,.
\end{equation}
Then
\begin{equation}\label{eq:maximal Lyapunov exp}
  \chi_\mu(T):=\lim_{n\to +\infty} \frac{1}{n}\int_X \log^+\mcll_n \dif\mu < \infty
\end{equation}
is well-defined, and for each $\delta>0$ we have:
\begin{equation}\label{eq:unstable dim lower}
  \dim_H \big(W_\delta^u(x,T)\big)\ge\frac{h_\mu(T)}{\chi_\mu(T)} \,,\quad \mu\text{-a.e.}~x\in X \,.
\end{equation}
\end{thm}

\begin{rmk}\mbox{}
\begin{enumerate}
  \item We call $\chi_\mu(T)$ defined by \eqref{eq:maximal Lyapunov exp} the {\it maximal Lyapunov exponent} of the measure-preserving system $(X,T,\mu)$, the pointwise version of which turns out to be equivalent to Kifer's notion of  maximal characteristic exponent introduced in \cite{Kif}. Unlike the situation in Theorem~\ref{thm:main C1 finite}, $h_\mu(T)>0$ cannot exclude the possibility of $\chi_\mu(T)=0$(a concrete example is provided in \S~\ref{subse:C1 infinite}) in general, and once this happens, \eqref{eq:unstable dim lower} simply means that $\dim_H \big( W_\delta^u(x,T) \big)=\infty$.
  \item If the assumption on $T$ in \eqref{eq:Lip cond} is replaced by $T^{-1}$ , then applying Theorem~\ref{thm:main TDS} to $T^{-1}$ we conclude that
\begin{equation}\label{eq:stable dim lower}
\dim_H \big( W_\delta^s(x,T) \big)\ge\frac{h_\mu(T^{-1})}{\chi_\mu(T^{-1})} = \frac{h_\mu(T)}{\chi_\mu(T^{-1})} \,,\quad \mu\text{-a.e.}~x\in X \,.
\end{equation}
\end{enumerate}
\end{rmk}
Our main result Theorem~\ref{thm:main C1 finite} follows from Theorem~\ref{thm:main TDS} immediately: for  any $C^1$ diffeomorphism $f$ on a Riemannian manifold $M$, $f^n$ is automatically a Lipschitz map for each $n\ge 1$ and it is evident that as a continuous function, the pointwise Lipschitz constant $\mcll_n(x)$ coincides with $\|\Dif f^n(x)\|$, and hence \eqref{eq:maximal Lyapunov exp} coincides with \eqref{eq:main C1 finite} for $f=T$. Therefore, in the rest of this paper we only need to prove Theorem~\ref{thm:main TDS}.

The key ingredient in our approach to Theorem~\ref{thm:main TDS} is to construct  a measurable partition subordinate to local stable/unstable sets and to establish local entropy formula for disintegration over such a partition. This intermediate step can be stated as the following theorem, where the notions on measurable partition and disintegration will be specified in \S~\ref{subse:partition} and the notations appearing in \eqref{eq:disintegrated local entropy} will be explained in Definition~\ref{def:FH}.

\begin{thm}\label{thm:local} Let $(X,T)$ be a TDS and let $\mu$ be an ergodic  $T$-invariant Borel probability  with $h_\mu(T)>0$. Then for any $\delta>0$, there exists a measurable partition $\xi$ of $X$ with the following properties.
\begin{enumerate}
\item $\overline{\xi(x)}\subset W_\delta^u(x,T)$ for  each $x\in X$, where $\xi(x)$ is the atom of $\xi$ containing $x$.

\item Let $\mu=\int_X\mu_x \dif \mu(x)$
be the measure disintegration of $\mu$ over $\xi$. Then for $\mu$-a.e. $x\in X$,
\begin{equation}\label{eq:disintegrated local entropy}
  \underline{h}_{\mu_x}(T,y)=\overline{h}_{\mu_x}(T,y)=h_\mu(T),\quad \mu_x\text{-a.e.}~y\in X\,.
\end{equation}
\end{enumerate}
\end{thm}

\begin{rmk}
  Let us say that the measurable partition appearing in the theorem above is {\it subordinate to} local unstable sets of $(X,T)$.  Applying the theorem to $T^{-1}$ instead of $T$, we can obtain another measurable partition subordinate to local stable sets of $(X,T)$ with analogous properties.
\end{rmk}

Let us give a partial list for literatures closely related to the main theme of this paper. For results relating dimension with entropy and Lyapunov exponents under $C^{1+\alpha}$ settings, see \cite{Yo,LY1,LY2,BPS}. For relevant results under $C^1$ or even TDS settings, see \cite{BHS,Su,Hu,FHYZ,HLY,HXY,WC,Y,HHW,FGH}.

The paper is organized as follows. \S~\ref{se:pre} is a preliminary section in which  we review some notions of ergodic theory. In \S~\ref{se:local} we construct a measurable partition subordinate to local unstable sets on whose fibres the local entropy concentrates, and prove Theorem~\ref{thm:local}. Based on this construction, in \S~\ref{se:dim} we complete the proof of Theorem~\ref{thm:main TDS}, and finally apply it to infinite dimensional $C^1$ dynamical systems.

\section{Preliminaries}\label{se:pre}

In this section, we review some notions in ergodic theory that will be used in subsequent sections. We shall restrict ourselves to talking about compact metric space $(X,d)$ and TDS $(X,T)$. For a TDS  $(X,T)$, let $\mscb_X$ denote its Borel $\sigma$-algebra and let
\[\msce_T:=\{E\in \mscb_X: T^{-1}E = E \} \,,\]
which is a sub-$\sigma$-algebra of $\mscb_X$ and a $\mscb_X$-measurable function $f$ on $X$ is $T$-invariant iff it is $\msce_T$-measurable. Recall that by our assumption, $T$ is a homeomorphism, so $\msce_T=\msce_{T^{-1}}$.

Let $\mclm(X)$, $\mclm(X,T)$ and  $\mclm^e(X,T)$ denote the collection of Borel measurable probability measures, $T$-invariant ones and $T$-ergodic ones respectively.  Given $\mu\in \mclm(X)$, let $\mscb_\mu$ be the completion of $\mscb_X$ with respect to $\mu$, so that $(X,\mscb_\mu)$ is a {\it Lebesgue space}(see, for example, \cite{Co,EW,Gl,Pa69,PrU,Ro}). Denote the collection of $\mu$-integrable functions on $X$ by $L^1(\mu)$. For $g\in L^1(\mu)$ and for any sub-$\sigma$-algebra $\mscc$ of $\mscb_\mu$, let $\mbbe_\mu(g|\mscc)$ be (a representative of) the conditional expectation of $g$ with respect to $\mu$ and $\mscc$.

\subsection{Ergodic theorems}\label{subse:ergodic thm}
We shall make use of two well known generalizations of Birkhoff's ergodic theorem listed below. The first one is a lemma named after Breiman as follows. See, for example \cite[Lemma 14.34]{Gl}), for a proof.

\begin{lem}[Breiman's Lemma]\label{lem:Breiman}
 Let $(X,T)$ be a TDS and let $\mu\in\mclm(X,T)$. Let $(g_n)$ be a sequence of measurable functions on $(X,\mscb_\mu)$  with the following properties:
 \begin{itemize}
   \item $\sup_n|g_n|\in L^1(\mu)$;
   \item there exists a measurable $g$ such that $\lim\limits_{n\to +\infty} g_n=g$ $\mu$-a.e.(and therefore in $L^1(\mu)$).
 \end{itemize}
Then we have:
\[ \lim_{N\to +\infty} \frac{1}{N}\sum_{n=0}^{N-1}g_n\circ T^n = \mbbe_\mu(g | \msce_T ) \quad \mu\text{-a.e. \ \ and \ \  in } L^1(\mu). \]
\end{lem}

Given a TDS $(X,T)$, a sequence of functions $\phi_n: X\to [-\infty,+\infty)$ is called {\it subadditive} w.r.t. $T$, if
\[\phi_{m+n}\le \phi_m+\phi_n\circ T^m, \quad \forall~m,n\ge 1.\]
The following subadditive ergodic theorem is first proved by Kingman in \cite{Kin}.
\begin{thm}[Kingman's subadditive ergodic theorem]\label{thm:Kingman}
Let $(X,T)$ be a TDS and let $\mu\in\mclm(X,T)$. Let $\phi_n$ be measurable and subadditive w.r.t. $T$ and suppose $\phi_1^+\in L^1(\mu)$. Then $\frac{1}{n}\phi_n$ converges  $\mu$-a.e. to some $T$-invariant measurable function $\psi: X\to [-\infty,+\infty)$, $\psi^+\in L^1(\mu)$ and
\[\int \psi\dif\mu =\lim_{n\to\infty}\frac{1}{n}\int\phi_n\dif\mu=\inf_{n\ge 1} \frac{1}{n}\int\phi_n\dif\mu\,.\]
Moreover, if $\phi_n\ge 0$, then the convergence is also in $L^1(\mu)$.
\end{thm}

\subsection{Measurable partition and disintegration}\label{subse:partition} For notions discussed in this subsection, please refer to, for example, \cite{Co,EW,Gl,Pa69,PrU,Ro}.
Let $X$ be a compact metric space. By definition, a (Borel) {\it partition} of $X$ is a collection of Borel sets in $X$ such that $X$ can be written as a disjoint union of them.  All the partitions appearing in this paper consist of Borel sets, so we omit ``Borel" and simply call them partitions. Elements in
a partition are called its {\it atoms}. Given a partition $\alpha$ of $X$ and $x\in X$, the unique atom in $\alpha$ that contains $x$ is denoted by $\alpha(x)$.
For two partitions $\alpha$ and $\beta$, we say that $\beta$ is {\it finer} than $\alpha$ or $\alpha$ is {\it coarser} than $\beta$, denoted by $\alpha\prec \beta$ or  $\beta\succ\alpha$, if each element in $\alpha$ is a union of a collection of elements in $\beta$. This defines a partial order on partitions. For a finite or countable family of partitions $(\alpha_i)_{i\in I}$, denote their common {\it refinement}
\[ \bigvee_{i\in I} \alpha_i :=\Big\{\bigcap_{i\in I} A_i : A_i\in\alpha_i,~\forall~i\in I \Big\} \,.\]
In other words, $\bigvee_{i\in I} \alpha_i$ is the coarsest partition that is finer than every $\alpha_i$. The meaning of $\alpha\vee\beta$ and $\bigvee_{i=1}^n \alpha_i$ are similar.

A partition $\xi$ of $X$ is called {\it measurable}, if it is countably generated in the following sense: there exist a sequence of finite partitions $(\alpha_n)_{n\ge 1}$ such that $\xi=\bigvee_{n=1}^\infty \alpha_n$. Given $\mu\in\mclm(X)$, for any measurable partition $\xi$, there exists an associated smallest $\mu$-complete sub-$\sigma$-algebra of $\mscb_\mu$, denoted by $\hxi$, which contains all the measurable sets that can be written as union of atoms in $\xi$. In other words,
\[\hxi=\{ A \in \mscb_\mu : A=\cup_{x\in A} \xi(x)  \}_\mu \,, \]
where $\mscc_\mu$ denotes the completion of $\mscc$ with respect to $\mu$ for any sub-$\sigma$-algebra $\mscc$ of $\mscb_\mu$. It is well known that for any sub-$\sigma$-algebra $\mscc$ of $\mscb_\mu$, there exists a measurable partition $\xi$ of $X$ such that $\mscc_\mu=\hxi$. For a measurable partition $\xi$, the conditional expectation $\mbbe_\mu(\cdot|\,\hxi\,)$ is also denoted by $\mbbe_\mu(\cdot| \xi)$. Properties of the {\it disintegration} of $\mu$ over $\xi$ that will be used in this paper are summarized in the proposition below(see, for example, \cite[Theorem~5.14]{EW} for a proof).

\begin{prop}\label{prop:disintegration}
Let $\xi$ be a measurable partition of $(X,d)$ and let $\mu\in \mclm(X)$. Then there exists a Borel set $X'$ of full measure and a family
$\{\mu_x\in \mclm(X): x\in X'\}$ which satisfy the following properties:
\begin{itemize}
  \item [(1)] For every $f \in L^1(X,\mscb_X,\mu)$, we have:
  \begin{equation}\label{eq:dis}
   \mbbe_\mu(f|\,\xi)(x):= \mbbe_\mu(f|\,\hxi\,)(x)=\int_X f\dif \mu_x \,, \quad \mu\text{-a.e.} ~x\in X \,.
  \end{equation}
  \item [(2)] For every $x\in X'$, $\mu_x(X'\cap \xi(x))=1$, and for $x,y \in  X'$,
  $\xi(x)=\xi(y)$ implies that $\mu_x=\mu_y$.
  \end{itemize}

\end{prop}

\subsection{Measure theoretic entropy}\label{subse:entropy}
We summarize some basic concepts and useful properties related to measure-theoretic entropy here. See, for example, \cite{Gl,Pa69,Pa04,PrU,Wa} for reference.

Let $X$ be a compact metric space and let $\mathcal{P}_X$ denote the collection of all its finite (Borel) partitions. Given $\alpha \in \mathcal{P}_X$, $\mu\in {\mathcal M}(X)$ and a sub-$\sigma$-algebra $\mscc\subset \mscb_\mu$, {\it the conditional information function} and {\it the conditional entropy} of $\alpha$ with respect to $\mscc$ are defined by:

$$I_\mu(\alpha|\mscc)(x):=-\sum_{A\in \alpha}1_{A}(x)\log \mathbb{E}_\mu(1_A|\mscc)(x)$$
and
$$H_{\mu}(\alpha|\mscc):=\int_X I_\mu(\alpha|\mscc)\dif\mu=\sum_{A\in \alpha} \int_X
-\mathbb{E}_\mu(1_A|\mscc) \log
\mathbb{E}_\mu(1_A|\mscc) \dif \mu,$$
where, as mentioned before, $\mathbb{E}_\mu(1_A|\mscc)$ is the conditional expectation of $1_A$(the indicator function of $A$) with respect to $\mscc$. One basic fact states that $H_\mu(\alpha| \mscc)$ increases in $\alpha$ and decreases in $\mscc$.

Now let $(X,T)$ be a TDS. Given a partition $\alpha$ of $X$ and integers $m<n$, denote
\[ \alpha_m^n:= \bigvee_{k=m}^n T^{-k}\alpha, \qtext{and} \alpha^+ :=\alpha_1^\infty=\bigvee_{k=1}^{+\infty} T^{-k}\alpha\,,\ \  \alpha^- :=\alpha_{-\infty}^{-1}=\bigvee_{k=1}^{+\infty} T^k\alpha \,. \]
When $\mu\in \mathcal{M}(X,T)$ and $\mscc$ is a $T$-invariant(i.e. $T^{-1}\mscc=\mscc$) sub-$\sigma$-algebra of $\mscb_\mu$ , it is not hard to see that $H_\mu(\alpha_0^{n-1}|\mscc)$ is a non-negative and subadditive sequence for any given $\alpha\in\mathcal{P}_X$, so
\[
h_\mu(T,\alpha|\mscc) := \lim_{n\rightarrow +\infty} \frac{1}{n} H_\mu(\alpha_0^{n-1} |\mscc)=\inf_{n\ge 1}
\frac{1}{n} H_\mu(\alpha_0^{n-1}|\mscc)
\]
is well defined.

For the trivial sub-$\sigma$-algebra $\mscc=\{
\emptyset, X\}$, we denote $H_{\mu}(\alpha|\mscc)$ and $h_\mu(T,\alpha|\mscc)$ by $H_{\mu}(\alpha)$ and $h_\mu(T,\alpha)$ respectively. {\it The
measure-theoretic entropy}  of $\mu$ is defined by
\[
h_\mu(T)=\sup_{\alpha \in \mathcal{P}_X} h_\mu(T,\alpha).
\]
A basic and useful estimate for upper bound of entropy is:
\[h_\mu(T,\alpha|\mscc)\le H_\mu(\alpha|\mscc)\big(\text{or } h_\mu(T,\alpha)\big) \le H_\mu(\alpha)\le \log\#\alpha \,.\]

The following elementary properties of information functions are well known (see, for example
\cite{Gl,Pa04}).

\begin{lem}\label{lem:info}
  Let $(X, T)$ be a TDS and let $\mu\in \mclm(X,T)$. Then we have:
  \begin{enumerate}
    \item (The information cocycle equation) For any $\alpha,\beta\in \mathcal{P}_X$ and any sub-$\sigma$-algebra $\mscc\subset \mscb_\mu$,
\[
I_\mu(\alpha\vee\beta|\mscc)=I_\mu(\alpha|\mscc)+I_\mu(\beta| \alpha\vee\mscc ).
\]

\item For any $\alpha\in \mathcal{P}_X$ and any sub-$\sigma$-algebra $\mscc\subset \mscb_\mu$,
\[
I_\mu(T^{-1}\alpha| T^{-1}\mscc)=I_\mu(\alpha|\mscc)\circ T \,.
\]
  \end{enumerate}
\end{lem}


Let $\mscf_n$ $(n\ge 1)$ and  $\mscf$ be sub-$\sigma$-algebras of $\mscb_\mu$.  Denote
\begin{itemize}
  \item $\mscf_n\nearrow \mscf$, if $\mscf_n$ is increasing in $n$ and $\mscf$ is the smallest $\sigma$-algebra containing each $\mscf_n$;
  \item $\mscf_n\searrow \mscf$, if $\mscf_n$ is decreasing in $n$ and $\mscf=\bigcap_{n=1}^\infty \mscf_n$.
\end{itemize}
The following martingale-type properties of information functions are also well known(\cite{Gl,Pa04}).

\begin{prop}\label{prop:martingale} Let $(X, T)$ be a TDS and let $\mu\in\mclm(X,T)$. Let $\mscf$, $\mscf_n$ $(n\ge 1)$ be sub-$\sigma$-algebras of $\mscb_\mu$ such that either $\mscf_n\nearrow \mscf$ or   $\mscf_n\searrow \mscf$ holds. Then for any $\alpha\in \mathcal{P}_X$, we have:
\begin{enumerate}
\item Chung's Lemma:
    \[
    f:=\sup_{n\ge 1} I_\mu(\alpha|\mscf_n) \in L^1(\mu) \textand \int_X f \dif \mu\le H_\mu(\alpha)+1 \,.
    \]

\item Martingale Theorem:
 \[
 \lim_{n\to\infty} I_\mu(\alpha|\mscf_n) = I_\mu(\alpha|\mscf) \quad \mu\text{-a.e. \ \ and \ \ in } L^1(\mu) \,.
 \]
\end{enumerate}
\end{prop}

The following is also a classical result (see for example \cite[Lemma~18.2]{Gl}).
\begin{lem}\label{lem:entropy asymp} Let $(X,T)$ be a TDS and let $\mu\in \mclm(X,T)$. If $\alpha,\talpha,\beta\in\mathcal{P}_X$ with $\alpha\prec\talpha$, then
\[
\lim \limits_{n\rightarrow+\infty} H_\mu \big(\alpha|\talpha^-\vee (T^{n}\beta)^-\big)=H_\mu(\alpha|\talpha^-).
\]
\end{lem}

\subsection{Local entropy}\label{subse:local entropy}
Let us recall measure-theoretic local entropy for general Borel probability measures introduced by Feng and Huang in \cite{FH}, which originates from the concept of Brin and Katok's local entropy for invariant measures  in \cite{BK}.

Let $(X,T)$ be a TDS. For $n \in \mbbn$, $x\in X$ and $r>0$,  let
\[
B_n(x,r , T):=\{ y\in X: \max_{0\le k \le n-1}  d(T^k x,T^k y)<r \}
\]
be the open {\it Bowen ball} centered at $x$ of step-length $n$  and radius $r$. We also denote it by $ B_n(x,r)$ for convenience when there is no risk of confusion for $T$.

Following Feng and Huang  \cite{FH}, we introduce:
\begin{defn}\label{def:FH} Given a TDS $(X,T)$ and
$\mu\in \mathcal{M}(X)$, the {\it measure-theoretic lower} and {\it upper entropy} of $\mu$ are defined respectively by
\begin{eqnarray*}
\underline{h}_{\mu}(T):=\int\underline{h}_{\mu}(T,x) \dif \mu(x), \qquad \overline{h}_{\mu}(T):=\int\overline{h}_{\mu}(T,x) \dif \mu(x),
\end{eqnarray*}
where
$$\underline{h}_{\mu}(T,x):=\lim_{\epsilon\rightarrow 0}\liminf_{n\rightarrow\infty}\frac{-\log
\mu(B_n(x,\epsilon,T))}{n},$$
$$\overline{h}_{\mu}(T,x):=\lim_{\epsilon\rightarrow 0}\limsup_{n\rightarrow\infty}\frac{-\log
\mu(B_n(x,\epsilon,T))}{n}.$$
\end{defn}
Recall that Brin and Katok \cite{BK} proved that for any $T$-invariant Borel probability measure $\mu$, $\underline{h}_{\mu}(T,x)=\overline{h}_{\mu}(T,x)$ for $\mu$-a.e. $x\in X$ is a $T$-invariant function, and
$$\underline{h}_{\mu}(T)=\overline{h}_{\mu}(T)=h_{\mu}(T).$$
In particular, when $\mu\in\mclm^e(X,T)$, $\underline{h}_{\mu}(T,x)=\overline{h}_{\mu}(T,x)=h_{\mu}(T)$ for $\mu$-a.e. $x\in X$.

\section{Measurable partition subordinate to local stable/unstable sets}\label{se:local}
%
 This section is denoted to the proof of Theorem~\ref{thm:local} and we divide it into three parts. In \S~\ref{subse:local SMB}, we present a local version of Shannon-McMillan-Breiman theorem for later use; in \S~\ref{subse:construct xi}, we construct $\xi$  explicitly and assertion (1) follows consequently; in \S~\ref{subse:local entropy proof} we prove assertion (2).

\subsection{A local version of Shannon-McMillan-Breiman theorem}\label{subse:local SMB}
The following statement can be seen as a local(conditional) version of the celebrated Shannon-McMillan-Breiman theorem.
\begin{prop}\label{prop:SMB local}
Let $(X,T)$ be a TDS and let $\mu\in \mclm(X,T)$. Let $\xi\subset \mscb_X$ be a measurable partition with $T\xi\prec \xi$ and let $\alpha\subset\mscb_X$ be a finite partition that satisfies $\alpha^-\prec \xi$. Denote

\[ \gamma_n:=\alpha^-\vee T^n\xi\,,~ \forall\,n\ge 0\qtext{and} \mscf:=\bigcap_{n=0}^\infty \hgamma_n \,.\]
Then we have:
\[
\lim_{N\to+\infty}\frac{1}{N} I_\mu\big( \alpha_0^{N-1} | \xi \big) = \mbbe_\mu\big(I_\mu(\alpha |\mscf) \big| \msce_T \big)\quad \mu \text{-a.e. \ \  and\ \  in } L^1(\mu)\,.
\]

\end{prop}

\begin{proof}
Define
\[ g_n=I_\mu \big( \alpha| \gamma_n \big)\,,~ \forall\,n\ge 0 \qtext{and} g=I_\mu(\alpha| \mscf ) \,.\]
From $T\xi\prec \xi$ we know that $\hgamma_n \searrow \mscf$. Thus by Chung's Lemma (see Proposition~\ref{prop:martingale} (1)),
\begin{equation}\label{eq:chung}
  \int_X \sup_{n\ge 0}g_n \dif\mu \le H_\mu(\alpha)+1<\infty \,;
\end{equation}
and by decreasing Martingale Theorem (see Proposition~\ref{prop:martingale} (2)),
\begin{equation}\label{eq:martingale}
  \lim_{n\to+\infty}g_n = g,\quad \mu \text{-a.e.\ \ and\ \ in } L^1(\mu) \,.
\end{equation}
\eqref{eq:chung} and \eqref{eq:martingale} enable us to apply Breiman's Lemma(see Lemma~\ref{lem:Breiman}) to $g_n$ and $g$, which implies that
\begin{equation}\label{eq:Breiman}
  \lim_{N\to+\infty}\frac{1}{N}\sum_{n=0}^{N-1}g_n\circ T^{n} = \mbbe_\mu(g | \msce_T ) , \quad \mu\text{-a.e. \ \ and \ in }L^1(\mu) \,.
\end{equation}
On the other hand, applying Lemma~\ref{lem:info} (1) repeatedly yields that
\[
I_\mu(\alpha_0^{N-1}|\xi) = I_\mu(\alpha|\xi) + I_\mu(\alpha_1^{N-1}|\alpha\vee\xi) = \cdots = I_\mu(\alpha|\xi)+ \sum_{n=1}^{N-1} I_\mu(T^{-n}\alpha|\alpha_0^{n-1}\vee\xi)\,.
\]
Then by Lemma~\ref{lem:info} (2), and noting that $\alpha^-\prec \xi$ implies $\gamma_n=\alpha_{-n}^{-1}\vee T^{n}\xi$ for every $n\ge 1$,  we obtain that
\[
I_\mu(\alpha_0^{N-1}|\xi) = I_\mu(\alpha|\xi)+ \sum_{n=1}^{N-1} I_\mu(\alpha|\alpha_{-n}^{-1}\vee T^{n}\xi)\circ T^{n} = \sum_{n=0}^{N-1}g_n\circ T^{n}\,.
\]
Combining this with \eqref{eq:Breiman}, the conclusion follows.
\end{proof}

\begin{rmk}
  In the statement of Proposition~\ref{prop:SMB local}, if we assume that $T\xi\succ \xi$  and $\alpha^-\succ \xi$ instead, then by replacing $\mscf$ with $\widehat{\alpha^-}$, the same conclusion holds and the proof is analogous. In this case, we can recover the classical Shannon-McMillan-Breiman theorem simply by taking $\xi=\{X\}$ to be the trivial partition.
\end{rmk}

As a direct corollary of Proposition~\ref{prop:SMB local}, we have:

\begin{cor}\label{cor:SMB local}
Under the settings in Proposition~\ref{prop:SMB local}, let us further assume that $\mu$ is ergodic. Let $\mu_x$ be the disintegration of $\mu$ over $\xi$ given in Proposition~\ref{prop:disintegration}. Then there exists $X_0\in \mscb_X$ with $\mu(X_0)=1$ such that for
\[ \Xi_x:=\xi(x)\cap X_0\,,~\forall\, x\in X_0 \qtext{and} \alpha_N:=\alpha_0^{N-1}\,,~\forall\, N\ge 1\,, \]
we have:

\begin{equation}\label{eq:ergodic SMB local}
\mu_x(\Xi_x)=1 \qtext{and} \lim_{N\rightarrow +\infty}\frac{-\log\mu_x \big( \alpha_N(y) \big)}{N}= H_\mu(\alpha|\mscf)\,,\ \ \forall\,x\in X_0,~\forall\,y\in \Xi_x \,.
\end{equation}

\end{cor}

\begin{proof}
Note that by definition,
\[
-\log\mu_x \big( \alpha_N(x) \big) = I_\mu( \alpha_N |\xi)(x), \quad \mu\text{-a.e. } x\in X\,.
\]
Also note that $\mbbe_\mu\big(I_\mu(\alpha |\mscf) \big| \msce_T \big) =H_\mu(\alpha|\mscf)$ holds $\mu$-a.e. since $\mu$ is assumed to be ergodic. Then  it follows from Proposition~\ref{prop:SMB local} that there exists $X_1\in \mscb_X$ with $\mu(X_1)=1$ such that
\[
\lim_{N\rightarrow +\infty}\frac{-\log\mu_y \big( \alpha_N(y) \big)}{N} = H_\mu(\alpha|\mscf)\,,\quad \forall\, y\in X_1\,.
\]
On the other hand, by Proposition~\ref{prop:disintegration} (2), there exists $X_2\in \mscb_X$ with  $\mu(X_2)=1$ such that for any $x\in X_2$,
$\mu_x(\xi(x)\cap X_2)=1$ and for any $x,y\in X_2$, $\xi(x)=\xi(y)$ implies that $\mu_x=\mu_y$. As a result, the following holds for any $N\ge 1$:
\[\mu_x \big( \alpha_N(y) \big) = \mu_y \big( \alpha_N(y) \big),\quad \forall\, x,y\in  X_2 \qtext{with} \xi(x)=\xi(y) \,.\]
Then for $X_0=X_1\cap X_2$, we have:
\[
\lim_{N\rightarrow +\infty}\frac{-\log\mu_x \big( \alpha_N(y) \big)}{N}= H_\mu(\alpha|\mscf),\quad \forall\, x\in X_0\,,~\forall\, y\in \xi(x)\cap X_0\,,
\]
which completes the proof.

\end{proof}

\subsection{Construction of $\xi$}\label{subse:construct xi}
Let $(X,T)$, $\mu$ and $\delta>0$ be as given in the statement of Theorem~\ref{thm:local}. The following construction of $\xi$ subordinate to the $\delta$-unstable sets is based on the construction in \cite{RS} or \cite[\S~6-2]{Pa69}.

Recall that for a measurable partition $\alpha$ of $X$,
$\alpha^-:=\bigvee_{n=1}^\infty T^{n}\alpha$. Denote $\partial\alpha:=\cup_{A\in\alpha}\partial A$. Let $\{\beta_i\}_{i=1}^{\infty}$ be a family of finite Borel partitions of $X$ with the following properties:
\begin{itemize}
  \item $\beta_i$ is finer and finer, i.e. $\beta_1\prec\beta_2\prec\beta_3\prec\cdots$;
  \item $\diam(\beta_1)\le\delta, \diam(T\beta_1)\le \delta$ and $\lim\limits_{i\rightarrow +\infty} \diam(\beta_i)=0$;
  \item $\mu(\partial\beta_i)=0$, $\forall~i\in \mathbb{N}$.
\end{itemize}
Starting from $k_1=0$, we can find a sequence of nonnegative integers $k_1,k_2,\cdots$ such that for $\alpha_q=\bigvee_{p=1}^q T^{k_p}\beta_p$, the following holds for every $q\ge 2$:

\begin{align}\label{eq:alpha_q}
H_{\mu}(\alpha_p |\alpha_{q-1}^-)-H_\mu(\alpha_p|\alpha_q^-)<
\frac{1}{p\cdot 2^{q-p}}\,,~ p=1,2,\cdots,q-1 \,.
\end{align}
The choice of such $k_q$ can be determined inductively on $q$ as follows. Once $k_p$ has been chosen for $1\le p<q$, applying Lemmea \ref{lem:entropy asymp} to $\alpha=\alpha_p$, $\talpha=\alpha_{q-1}$ and $\beta=\beta_q$ for each $1\le p\le q-1$, we can find $k_q$ such that \eqref{eq:alpha_q} holds.

Then we can specify our choice of $\xi$:
\[
\mathcal{P}:=\bigvee_{p=1}^\infty \alpha_p \qtext{and} \xi:=\mathcal{P}^- \,.
\]
Given $p,m\in \mathbb{N}$, taking sum on both sides of \eqref{eq:alpha_q} for $q=p+1,\cdots, p+m$  yields that
$$H_\mu(\alpha_p|\alpha_p^-)-H_\mu(\alpha_p|\alpha_{p+m}^-)<\frac{1}{p}\,.$$
Let $m\rightarrow +\infty$, one has
$$H_\mu(\alpha_p|\xi)\ge H_\mu(\alpha_p|\alpha_p^-)-\frac{1}{p}=h_\mu(T,\alpha_p)-\frac{1}{p}.$$
Define
$$c_p=H_\mu\Big(\alpha_p\big|\bigcap_{n=0}^{+\infty}\widehat{\gamma_{p,n}}\Big)\,,\quad\text{where }\  \gamma_{p,n}:=\alpha_p^-\vee T^{n}\xi\,.$$
Since $\alpha_p^-\prec \gamma_{p,n}\prec \xi$ for each $n$, $ H_\mu(\alpha_p|\xi)\le c_p\le H_\mu(\alpha_p|\alpha_p^-)$, and hence
$$h_\mu(T,\alpha_p)-\frac{1}{p}\le c_p\le h_\mu(T,\alpha_p).$$
Then
\begin{align}
\sup_{p\ge 1} c_p=\sup_{p\ge 1}h_\mu(T,\alpha_p)=h_\mu(T).
\end{align}

Next let us show that $\overline{\xi(x)}\subset W_\delta^u(x,T)$ for  each $x\in X$. Given $x\in X$, note that
\[
\overline{\xi(x)}\subset \overline{(T\beta_1)(x)} \qtext{and} T^{-i}(\overline{\xi(x)})=\overline{T^{-i}\big( \xi(x)\big)}=\overline{(T^{-i}\xi)(T^{-i}x)}
\subset \overline{\beta_1(T^{-i}x)} \,,\ \forall\,i\ge 1\,.
\]
Thus
\begin{align}\label{diam-upper-1}
\diam\Big( T^{-i}\big( \overline{\xi(x)}\big)\Big)\le \max\{\diam(T\beta_1),\diam(\beta_1)\}\le \delta \,,\ \forall\, i\ge 0\,.
\end{align}
For any $j\ge 1$ and $i\ge 0$, one has
\begin{align*}
T^{-(k_j+i)}\big( \overline{\xi(x)} \big)=\overline{T^{-(k_j+i)}\big( \xi(x)\big)}=\overline{(T^{-(k_j+i)}\xi)(T^{-(k_j+i)}x)}
\subset \overline{\beta_j(T^{-(k_j+i)}(x))}\,.
\end{align*}
Thus
\begin{align}\label{diam-upper-2}
\diam\Big( T^{-(k_j+i)}\big(\overline{\xi(x)}\big)\Big) \le \diam(\beta_j),\quad\forall\,j\ge 1,~i\ge0\,.
\end{align}
Due to \eqref{diam-upper-1}, \eqref{diam-upper-2} and $\lim_{j\to\infty}\diam(\beta_j)=0$, $\overline{\xi(x)}\subset W_\delta^u(x,T)$. The proof of assertion (1) is done.

\subsection{Proof of local entropy formula}\label{subse:local entropy proof}

In this subsection we prove assertion (2) in Theorem \ref{thm:local}. For preparation  we need the following result.

\begin{lem}\label{lem:shift}
Let $(X,T)$ be a TDS and let $\mu\in \mathcal{M}(X)$. Let $\alpha$ and $\beta_1\prec\cdots\prec \beta_n \prec\cdots$ be finite partitions of $X$. If there exists a constant $c$ such that
$$\lim_{n \to  +\infty} \frac{-\log \mu\big( (\alpha\bigvee \beta_n)(x)\big)} {n}=c, \ \ \mu\text{-a.e.}~x\in X \,,$$
then the following holds:
$$\lim_{n \to +\infty} \frac{-\log \mu\big(  \beta_n(x) \big)} {n}=c,\quad \mu\text{-a.e.}~x\in X\,.$$
In particular, for any finite partition $\alpha$ and any $k\ge 1$, we have:
\[ \lim_{n \to  +\infty} \frac{-\log \mu\big( \alpha_{0}^{n}(x)\big)} {n}=c, \ \ \mu\text{-a.e.}~x\in X  \implies  \lim_{n \to  +\infty} \frac{-\log \mu\big( \alpha_{k}^{n}(x)\big)} {n}=c, \quad \mu\text{-a.e.}~x\in X \,.\]
\end{lem}

\begin{proof}
Noting that
\[\limsup_{n\to +\infty} \frac{-\log \mu\big( \beta_n(x)\big)} {n}\le \limsup_{n\to +\infty} \frac{-\log \mu\big( (\alpha\bigvee \beta_n)(x)\big)} {n}\,,\]
it suffices to show that for an arbitrary fixed constant $\lambda<c$, we have
\[
\mu(E)=0 \qtext{for} E:=\{x\in X: \mu\big( \beta_n(x)\big)\ge e^{-\lambda n} \text{ holds for infinitely many } n\}.
\]
To this end, for each $n\ge 1$ and each $A\in \alpha$, let $\mclb(n,A)$ be the collection of set $B\in \beta_n$ satisfying the following two properties:
\begin{itemize}
  \item [(a)] $\mu(A\cap B)\ge\mu(A'\cap B)$, $\forall~A'\in\alpha$;
  \item [(b)] $\mu(A\cap B)\ge \frac{1}{\#\alpha}e^{-\lambda n}$.
\end{itemize}
Let $F_{n,A}=\cup_{B\in \mclb(n,A)} B$ and let $E_A=\cap_{m=1}^\infty\cup_{n=m}^\infty F_{n,A}$. We claim that $E\subset \cup_{A\in \alpha}E_A$. To see this, pick an arbitrary $x\in E$. Then  there exists an infinite set $S\subset\mbbn$ such that $\mu\big( \beta_n(x)\big)\ge e^{-\lambda n}$  holds for  $n\in S$. Since $\#\alpha$ is finite, by pigeonhole principle, there exists $A\in\alpha$ and an infinite subset $S'$ of $S$ such that for each $n\in S'$, $A$ and $B=\beta_n(x)$ satisfy property (a). Property (b) also follows for such a pair of $A$ and $B$. By the definition of $E_A$, $x\in E_A$, which verifies the claim.

To complete the proof, we only need to show that $\mu(E_A)=0$ for each $A\in\alpha$. Due to property (b) and the definition of $E_A$,
\[A\cap E_A\subset\{x\in A: \mu\big( (\alpha\vee \beta_n)(x)\big)\ge \frac{1}{\#\alpha}e^{-\lambda n} \text{ holds for infinitely many } n\}\,.\]
Then by the assumption in the lemma, $\mu(A\cap E_A)=0$. Note that for each $m\ge 1$, $\cup_{n=m}^\infty F_{n,A}$ can be written as a disjoint union $\sqcup_{B\in \mclc }B$, where $\mclc=\mclc(m,A)$ is a subset of  $\cup_{n=m}^\infty \mclb(n,A)$. On the other hand, for each $B\in \mclc$, $\mu(B)\le \#\alpha\cdot\mu(A\cap B)$. Taking sum over all $B\in \mclc$ on both sides yields that
\[ \mu(\cup_{n=m}^\infty F_{n,A} )\le \#\alpha\cdot \mu(A\cap \cup_{n=m}^\infty F_{n,A} )\,.\]
Letting $m\to\infty$, it follows that $\mu(E_A)\le \#\alpha\cdot\mu(A\cap E_A)=0$. The conclusion follows.
\end{proof}

Now we are ready to prove assertion (2) in Theorem~\ref{thm:local}. According to Corollary~\ref{cor:SMB local}, the following holds for each $j\ge 1$:
\[
\lim_{N\to+\infty}\frac{-\log\mu_x \big( (\bigvee\limits_{i=0}^{N-1}T^{-i} \alpha_j)(y) \big)}{N}=c_j \,,  \quad \mu_x\text{-a.e.}~y\in X,\quad \mu\text{-a.e.}~x\in X\,.
\]
Then by Lemma~\ref{lem:shift}, it follows that
\begin{equation}\label{eq:SMB local shift}
  \lim_{N\to+\infty}\frac{-\log\mu_x \big( (\bigvee\limits_{i=0}^{N-1}(T^{-(k_j+i)}\alpha_j)(y) \big)}{N}=c_j \,,  \quad \mu_x\text{-a.e.}~y\in X,\quad \mu\text{-a.e.}~x\in X\,.
\end{equation}
Given $\epsilon>0$, we can find $N=N(\epsilon)\in \mathbb{N}$ such that $\diam(\beta_j)\le \epsilon$ for $j\ge N$, as $\lim\limits_{i\rightarrow +\infty}\diam(\beta_i)=0$. Note that $T^{-k_j}\alpha_j\succ \beta_j$. Hence when $j\ge N$, for any $n\ge 1$ we have $\big(\bigvee_{i=0}^{n-1}T^{-(k_j+i)}\alpha_j\big)(y)\subset B_n(y,\epsilon,T)$, and consequently
\[
\frac{-\log \mu_x (B_n(y,\epsilon,T))}{n}\le \frac{-\log \mu_x\big( \bigvee_{i=0}^{n-1}(T^{-(k_j+i)}\alpha_j)(y) \big)}{n}\,,\quad \mu\text{-a.e.}~x\in X, \forall y\in X\,.
\]
Thus for $\mu$-a.e. $x\in X$ and any $y\in X$, we have:
\begin{align}\label{eq:upper local entropy}
\overline{h}_{\mu_x}(T,y)&= \lim_{\epsilon\rightarrow 0}\limsup_{n\rightarrow+\infty}\frac{-\log\mu_x\big( B_n(y,\epsilon,T)\big)}{n}\nonumber\\
&\le \sup_{j\ge 1} \limsup_{n\rightarrow+\infty}\frac{-\log\mu_x \big( (\bigvee_{i=0}^{n-1}(T^{-(k_j+i)}\alpha_j)(y) \big)}{n}.
\end{align}
Combining \eqref{eq:SMB local shift} and \eqref{eq:upper local entropy}, we obtain that:
\[
\overline{h}_{\mu_x}(T,y) \le \sup_{j\ge 1} c_j=h_\mu(T)\,, \quad \mu_x\text{-a.e.}~y\in X,\quad \mu\text{-a.e.}~x\in X\,.
\]

To complete the proof, it remains to show that for $\mu$-a.e. $x\in X$, $\underline{h}_{\mu_x}(T,y) \ge h_\mu(T)$ holds for $\mu_x$-a.e. $y\in X$.
Since $\sup_{j\ge 1}c_j=h_\mu(T)$, it suffices to prove that for any $j\in\mathbb{N}$ and $\mu$-a.e. $x\in X$, $\underline{h}_{\mu_x}(T, y)\ge c_j$ holds for $\mu_x$-a.e. $y\in X$. Due to Corollary~\ref{cor:SMB local} and the definition of $c_j$(recall that $c_j=H_\mu(\alpha_j|\xi)$ and $\mu(\partial\alpha_j)=0$), the problem is reduced to verifying the lemma below.

\begin{lem}\label{lem:local entropy lower bound}
 Under the settings in Corollary~\ref{cor:SMB local}, assume that $\mu(\partial\alpha)=0$ additionally. Then for $\mu$-a.e. $x\in X$,  $\underline{h}_{\mu_x}(T, y)\ge H_\mu(\alpha|\mscf)$ holds for $\mu_x$-a.e. $y\in X$.
\end{lem}


The lemma above follows from the argument of Brin and Katok \cite{BK} in  bounding their lower local entropy from below. For completion, we give a self-contained proof of it in the appendix. This completes the proof of Theorem~\ref{thm:local}.

\section{Lower bound of Hausdorff dimension}\label{se:dim}

In this section, we shall complete the proof Theorem~\ref{thm:main TDS} in \S~\ref{subse:main thm proof}. For preparation, we need an alternative definition of maximal Lyapunov exponent introduced by Kifer \cite{Kif} and show the equivalence between two definitions in \S~\ref{subse:max exp}. Finally, we apply Theorem~\ref{thm:main TDS} to infinitely dimensional $C^1$ systems in \S~\ref{subse:C1 infinite}.

\subsection{On maximal Lyapunov exponent}\label{subse:max exp}  Following \cite{Kif}, define(here we still set $\sup\varnothing =0$)
\[
\mcll_n^r(x):=\sup_{y\in B_n(x,r)\setminus\{x\}}\frac{d(T^n x,T^n y)}{d(x,y)} \in [0,+\infty].
\]
By continuity of $T$ and openness of $B_n(x,r)$, it is easy to see that $\mcll_n^r$ is lower-semicontinuous.  Recall $\mcll_n$ defined in \eqref{eq:Lip const point} and note that by continuity, for each $n\ge 1$, $\mcll_n^r$ monotonically converges to $\mcll_n$ as $r\searrow 0$. It follows that $\mcll_n$ is a Borel function, and therefore $\int_X \log^+\mcll_n \dif\mu$ in \eqref{eq:maximal Lyapunov exp} is well-defined. Moreover, as observed by Kifer \cite{Kif}, for each $r>0$, $\log^+\mcll_n^r$ is a subadditive sequence of functions in $n$ w.r.t. $T$. To proceed, let us make the following simple observation:

\begin{lem}\label{lem:mono subadd}
Let $(X,T)$ be a TDS and let $\mu\in\mclm(X,T)$. Let $\phi_n^m\ge 0$ be measurable functions on $X$ with $\phi_1^1\in L^1(\mu)$. Suppose that for each $m$, $\phi_n^m$ is subadditive in $n$ w.r.t. $T$; and for each $n$, $\phi_n^m$ is decreasing in $m$. Then the following two iterated limits converge in both $\mu$-a.e. and $L^1(\mu)$ sense, and the order is exchangeable:
  \[
  \lim_{m\to\infty} \lim_{n\to\infty}\frac{1}{n}\phi_n^m = \lim_{n\to\infty} \lim_{m\to\infty}\frac{1}{n}\phi_n^m\,.
  \]
\end{lem}
\begin{proof}
Since $0\le \phi_1^m\le \phi_1^1\in L^1(\mu)$, by subadditive ergodic theorem, $\psi^m:=\lim\limits_{n\to\infty}\frac{1}{n}\phi_n^m$ converges $\mu$-a.e. and in $L^1(\mu)$; moreover, for each $m$ we have
\[
0\le \int\psi^m\dif\mu =\inf_n\frac{1}{n}\int\phi_n^m\dif\mu\le \int\phi_1^1\dif\mu<\infty\,.
\]
Noting that $0\le \psi^m\le \psi^1\in L^1(\mu)$ is decreasing in $m$, $\psi:=\lim\limits_{m\to\infty}\psi^m$ also converges $\mu$-a.e. and in $L^1(\mu)$. On the other hand, for the same reason, $\varphi_n:=\lim\limits_{m\to\infty}\phi_n^m$ converges $\mu$-a.e. and in $L^1(\mu)$. As pointwise limit of subaddtive sequences, $\varphi_n\ge 0$ is still subadditive and $\varphi_1\in L^1(\mu)$, so that $\Psi:=\lim\limits_{n\to\infty}\frac{1}{n}\varphi_n$  converges $\mu$-a.e. and in $L^1(\mu)$; moreover, $0\le \int\Psi\dif\mu =\inf\limits_n\frac{1}{n}\int \varphi_n\dif\mu<\infty$.

By definition, $\frac{1}{n}\varphi_n\le \frac{1}{n}\phi_n^m$  $\mu$-a.e. for any $m,n$. First letting $n\to\infty$  and then letting $m\to\infty$ yields that $\Psi\le\psi$ $\mu$-a.e.. To complete the proof, it suffices to show that $\int\Psi\dif\mu\ge \int\psi\dif\mu$. To this end, observe that
\[
\int \psi\dif\mu \le \int\psi^m\dif\mu \le \frac{1}{n}\int\phi_n^m\dif\mu\,,\quad\forall m,n\,.
\]
First letting $m\to\infty$ and then letting $n\to\infty$, the conclusion follows.
\end{proof}
As a direct corollary of Lemma~\ref{lem:mono subadd} together with subadditivity of $\log^+\mcll_n^r$, we have:

\begin{prop}\label{prop:Kifer}
Let $(X,T)$ be a TDS and let $\mu\in\mclm(X,T)$. Assume that $\log^+\mcll_1^{r_*}\in L^1(\mu)$ for some $r_*>0$. Then all the limits below converge in both $\mu$-a.e. and $L^1(\mu)$ sense:
\[
\chi(x):= \lim_{n\to\infty}\frac{1}{n}\log^+\mcll_n(x),\quad \Lambda^r(x):=\lim_{n\to\infty}\frac{1}{n}\log^+\mcll_n^r(x),~ r\in(0,r_*] \,.
\]
Moreover, for any sequence $r_m\searrow 0$, the following equality also holds $\mu$-a.e. and in $L^1(\mu)$:
\[
\chi(x)=\lim_{m\to\infty}\Lambda^{r_m}(x)\,.
\]
\end{prop}

\begin{proof}
Given an arbitrary sequence $r_m\searrow 0$ starting with $r_1\le r_*$, the family $\phi_n^m:=\log^+\mcll_n^{r_m}$ satisfy all the assumptions in Lemma~\ref{lem:mono subadd}. All the assertions follow.
\end{proof}
\begin{rmk}
  We call $\chi(x)$ the {\it maximal Lyapunov exponent} at $x$ provided the limit exists. The proposition guarantees that under assumption \eqref{eq:Lip cond}, $\chi_\mu(T)$ appearing in Theorem~\ref{thm:main TDS} satisfies $\chi_\mu(T)=\int_X\chi(x)\dif\mu<\infty$; in particular, in the ergodic case, $\chi(x)=\chi_\mu(T)$ for $\mu$-a.e. $x\in X$. Besides, the proposition says that $\chi(x)$ coincides with the notion of characteristic maximal exponent introduced by Kifer \cite{Kif}.
\end{rmk}

\subsection{Proof of Theorem~\ref{thm:main TDS}}\label{subse:main thm proof} Let us begin with another simple observation:

\begin{lem}\label{lem:ball}
Following the settings and notations in Proposition~\ref{prop:Kifer}, the statement below holds for $\mu$-a.e. $x\in X$ with $\chi(x)<\infty$.
Given $\lambda>\chi(x)$ and $\epsilon>0$, there exist $\eta>0$ and $N\ge 1$ such that $B(x,\frac{\eta}{e^{n\lambda}})\subset B_n(x,\epsilon)$ for any $n\ge N$.
\end{lem}

\begin{proof}
According to Proposition~\ref{prop:Kifer}, the following holds for $\mu$-a.e. $x\in X$:
$$\chi(x)=\lim_{m\to\infty}\Lambda^{r_m}(x)<\infty;$$
so given $\lambda>\chi(x)$ and $\epsilon>0$,
there exists $r\in (0,\epsilon]$ such that $\Lambda^r(x)<\lambda-c$, where $c:=\frac{1}{2}(\lambda-\chi(x))>0$.  It suffices to prove the statement for such an $x$ and for $r$ instead of $\epsilon$. Let $N_0\ge 1$ with $e^{-c N_0}<r$.
By definition of $\Lambda^r(x)$, there exists $N\ge N_0$ such that $\mcll_n^r(x)<e^{(\lambda-c)n}$ when $n\ge N$.  By continuity, we can find $\eta\in (0,1)$ such that $B(x,\frac{\eta}{e^{\lambda N}})\subset B_N(x,r)$, so the conclusion holds for $n=N$. By induction, we may suppose that $B(x,\frac{\eta}{e^{\lambda n}})\subset B_n(x,r)$ for some $n\ge N$. It follows that
\[
y\in B(x,\frac{\eta}{e^{\lambda (n+1)}}) \subset B(x,\frac{\eta}{e^{\lambda n}}) \subset B_n(x,r) \implies d(T^n x, T^n y)\le \mcll_n^r(x)\cdot d(x,y)< e^{(\lambda-c)n}\cdot \frac{\eta}{e^{\lambda n}} < e^{-c N_0}<r \,.
\]
Thus we obtain that $B(x,\frac{\eta}{e^{\lambda (n+1)}}) \subset B_{n+1}(x,r)$, which completes the induction.
\end{proof}

Now we are ready to prove Theorem~\ref{thm:main TDS}.  $\chi_\mu(T)$ has been shown well-defined and finite-valued by Proposition~\ref{prop:Kifer}. For the rest, fixing $\delta>0$, let $\xi$ be the measurable partition subordinate to $\delta$-local unstable sets asserted in Theorem~\ref{thm:local}~(1). Let $\{\mu_x\in \mclm(X): x\in X'\}$ be the disintegration of $\mu$ over $\xi$ as stated in Proposition~\ref{prop:disintegration}. Firstly, observe that for each $x\in X'$, the {\it lower local dimension} of $\mu_x$ below satisfies that
\begin{equation}\label{eq:local dim}
\underline{d}_{\mu_x}(y):=\liminf_{r\searrow 0}\frac{\log \mu_x \big(B(y,r)\big)}{\log r} = \liminf_{n\to \infty}\frac{-\log \mu_x \big( B(y,e^{-n\lambda}) \big)}{ n\lambda},\quad \forall y\in X,~\lambda>0\,.
\end{equation}
Secondly, note that  $\chi(y)=\chi_\mu(T)$ holds for $\mu$-a.e. $y\in X$ because  $\mu$ is assumed to be ergodic. Then Lemma~\ref{lem:ball} implies that
there exists $Y\subset X$ of full measure with the following properties: for each $y\in Y$, $\chi(y)=\chi_\mu(T)$; moreover, given $\lambda>\chi_\mu(T)$ and $\epsilon>0$, there exists $\eta>0$ such that
\[
\liminf_{n\to \infty}\frac{-\log \mu_x \big( B(y,\frac{\eta}{e^{n\lambda}}) \big)}{ n} \ge \liminf_{n\to \infty}\frac{-\log \mu_x \big( B_n(y,\epsilon,T)\big)}{n},\quad \forall x\in X',y\in Y\,.
\]
Since the left hand side of the above inequality is independent of $\eta$, taking $\eta=1$ and letting $\epsilon\to 0^+$ yields that
\begin{equation}\label{eq:lower local entropy}
  \liminf_{n\to \infty}\frac{-\log \mu_x(B(y,e^{-n\lambda}))}{ n} \ge \lim_{\epsilon\to 0^+}\liminf_{n\to \infty}\frac{-\log \mu_x(B_n(y,\epsilon,T))}{ n} =\underline{h}_{\mu_x}(T,y),\quad \forall x\in X',y\in Y\,.
\end{equation}
Thirdly, combining \eqref{eq:local dim} and \eqref{eq:lower local entropy} with Theorem~\ref{thm:local}~(2), we conclude that for any $\lambda>\chi_\mu(T)$, the following holds for $\mu$-a.e. $x\in X$:

\[\underline{d}_{\mu_x}(y)\ge \frac{\underline{h}_{\mu_x}(T,y)}{\lambda}  = \frac{h_\mu(T)}{\lambda},\quad \mu_x\text{-a.e.}~y\in\xi(x)\cap Y\,.\]
Since $\mu_x( \xi(x)\cap Y)=1$ for $\mu$-a.e. $x\in X$, from the inequality above and the {\it mass distribution principle}(see, for example, \cite{Pe97,PrU}), we obtain that
\[
\dim_H( W_\delta^u(x))\ge \dim_H( \xi(x)\cap Y) \ge \frac{h_\mu(T)}{\lambda}\,,\quad \mu\text{-a.e.}~x\in X\,.
\]
Since $\lambda>\chi_\mu(T)$ is arbitrary, the proof of Theorem~\ref{thm:main TDS} is completed.

\subsection{On infinitely dimensional $C^1$ systems}\label{subse:C1 infinite}

As $M$ is of finite dimension is inessential in our proof of Theorem~\ref{thm:main C1 finite} for $\delta$-unstable set, an analogous conclusion in infinite dimensional case also holds for $\delta$-unstable set. More precisely, let $B$ be a Banach space, let $U\subset B$ be an open subset and let $f:U\to B$ be a $C^1$ map. Suppose that  $X\subset U$ is compact and $f$ maps $X$ homeomorphically onto itself. Then for $T:=f|_X$, $(X,T)$ is a TDS, where the metric $d$ on $X$ is induced by the norm $\|\cdot\|$ on $B$. Given $x\in X$ and $r>0$, denote open ball centered at $x$ of radius $r$ in $B$(respectively $X$) by $O(x,r)=\{y\in B:\|x-y\|<r\}$(respectively $B(x,r)=O(x,r)\cap X$). When $O(x,r)\subset U$, by mean-value theorem and convexity of $O(x,r)$,
\[
\mcll_1^r(x)=\sup_{y\in B(x,r)\setminus\{x\}}\frac{d(Tx,Ty)}{d(x,y)}\le \sup_{y\in O(x,r)}\|Df(y)\| \,.
\]
This fact together with continuity of $\|Df\|$ on $U$ and compactness of $X$ implies that $\mcll_1^r$ is bounded from above on $X$ for some $r>0$, so
the integrability condition \eqref{eq:Lip cond} is automatically satisfied for this $(X,T)$. Therefore, we have:

\begin{thm}\label{thm:C1 infinite}
Let $B$, $f$, $X$ and $T$ be as above and let $\mu\in\mclm^e(X,T)$ be of positive entropy. Then the following holds for every $\delta>0$:

\[
\dim_H(W_\delta^u(x,T))\ge\frac{h_\mu(T)}{\chi_\mu(T)} \,,\quad \mu\text{-a.e.}~x\in X \,.
\]
\end{thm}

\begin{rmk}\mbox{}
\begin{enumerate}
  \item Evidently, for either finite dimensional or infinite dimensional case, once $f^{-1}$ exists and is $C^1$-map  on an open neighborhood of $X$, similar lower bound estimate for Hausdorff dimension of local stable sets holds.
  \item It might be mentioned that for infinite dimensional systems, it could happen that $h_\mu(T)>0$  while $\chi_\mu(T)=0$. The following is such an example.
\end{enumerate}
\end{rmk}

\begin{exa}
  There exist a Hilbert space $\mbbh$, a bounded linear operator $T$ on $\mbbh$ of spectrum radius $1$, a compact $T$-invariant $X\subset\mbbh$ and $\mu\in\mclm^e(X,T)$ with positive entropy.
\end{exa}

\begin{proof}

 Given  a sequence $\mathbf{a}=(a_k)_{k\ge 0}$ of positive reals, denote
\[
\mbbh_{\mathbf{a}}:=\big\{(x_n)_{n\in \mbbz}\in\mbbr^{\mbbz} \mid \sum_{n\in \mbbz}a_{|n|}|x_n|^2<+\infty\big\}\,.
\]
Then $\mbbh_{\mathbf{a}}$ is a Hilbert space over $\mbbr$ with respect to the inner product below:
\[
\langle x, y\rangle_{\mathbf{a}}:=\sum_{n\in \mbbz} a_{|n|} x_ny_n \qtext{for} x=(x_n)_{n\in \mbbz}\,,~y=(y_n)_{n\in \mbbz}\in \mbbh_{\mathbf{a}}\,.
\]
Let $\|\cdot\|_{\mathbf{a}}$ denote the norm on $\mbbh_{\mathbf{a}}$ induced by its inner product.
Let us say that the sequence $\mathbf{a}=(a_k)_{k\ge 0}$ above {\it sub-exponentially decreases to $0$}, if it satisfies the following two properties:
\begin{itemize}
  \item $a_k$ is strictly decreasing to $0$ as $k\to\infty$;
  \item there exist $C>0$ and a sequence $\mathbf{b}=(b_k)_{k\ge 0}$  of positive reals with $\lim\limits_{k\to\infty}\frac{1}{k}|\log b_k|=0$ such that
  \[
  \frac{a_k}{a_l}\le C b_{|k-l|},\quad \forall~k,l\ge 0\,.
  \]
\end{itemize}

From now on suppose that $\mathbf{a}=(a_k)_{k\ge 0}$ sub-exponentially decreases to $0$. We claim that the left-shift map $T$ acting on $\mbbh_{\mathbf{a}}$ is a bounded linear operator with spectrum radius $1$. The linearity of $T$ and the fact that $\|T^k\|_{\mathbf{a}}>1$ for each $k\ge 1$ are evident. On the other hand, for $x=(x_n)_{n\in \mbbz}\in \mbbh_{\mathbf{a}}$,
\[
\|T^k x\|_{\mathbf{a}}^2 =\sum_{n\in\mbbz}a_{|n-k|}|x_n|^2 \le C\cdot \max\{b_0,\cdots, b_k\}\cdot\|x\|_{\mathbf{a}}^2\,,
\]
which implies that the spectrum radius of $T$ is bounded by $1$ from above.

Now further assume that $\sum_{k=0}^\infty a_k<\infty$; for example, let $a_k=\frac{1}{k^2+1}$. Then $X=\{0,1\}^\mbbz$ can be identified as a subset of $\mbbh_{\mathbf{a}}$ in the natural way. The norm on $\mbbh_{\mathbf{a}}$ induces a metric on $X$ that makes $X$ a compact metric space, which is compatible with the standard product topology on $X$. Let $\mu=(\frac{1}{2}\delta_0+\frac{1}{2}\delta_1)^{\mathbb{Z}}$. Then $\mu\in\mclm^e(X,T)$ with $h_\mu(T)=\log2$.
Since the spectrum radius of $T:\mbbh_{\mathbf{a}}\to \mbbh_{\mathbf{a}}$ is $1$, the maximal Lyapunov exponent $\chi_\mu(T)$ of $T:X\to X$ is $0$.
\end{proof}

\appendix

\section{Proof of Lemma~\ref{lem:local entropy lower bound}}

Following the argument of Brin and Katok \cite{BK}, it suffices to prove the technical statement below, where we denote $c=H_\mu(\alpha|\mscf)$.

\begin{clm}
For any $\epsilon>0$ sufficiently small, there exist $\tau>0$ and a measurable subset $I$ of $X$ with $\mu(I)>1-\epsilon^{\frac{1}{4}}$ for which the following statement holds. For any $x\in I$, we can find a measurable subset $D$ of $X$ such that
\begin{equation}\label{eq:claim}
 \mu_{x}(D)>1-4\epsilon^{\frac{1}{4}}\,, \qtext{and} \liminf_{n\to+\infty}\frac{-\log\mu_{x}\big(B_n(y,\tau,T)\big)}{n}\ge c-3(\Delta+\epsilon)\,,\ \forall y\in D\,,
\end{equation}
where
\begin{equation}\label{eq:Stirling const}
  \Delta:=2\sqrt{\epsilon}\log (\#\alpha-1)-2\sqrt{\epsilon}\log2\sqrt{\epsilon}-(1-2\sqrt{\epsilon})\log(1-2\sqrt{\epsilon})\,.
\end{equation}
\end{clm}
Since the proof is a little bit long, let us provide an outline here briefly. Firstly, we construct the set $I$ explicitly. Secondly,
given $x\in I$, the associated set $D$ will be taken with the form $E\setminus \cup_{n=\ell}^\infty E_n$; to define $E_n$, we need to introduce a pseudo-metric $\rho_n$ on $\alpha_n$ in advance. Finally we show that $D$ satisfies the desired properties.

\begin{proof}[Proof of Claim]
Let $\epsilon>0$ be a small number whose upper bound will be specified in \eqref{eq:Striling upper}. We may suppose $\# \alpha \ge 2$ because otherwise there is nothing to prove.

For $r>0$ and $x\in X$, let $B(x,r)=\{y\in X:d(x,y)<r\}$ and
\[
V_r(\alpha)=\{x\in X: B(x,r)\nsubseteq \alpha(x)\}.
\]
Since $\bigcap\limits_{r>0}V_r(\alpha)=\partial \alpha$, we have $\lim_{r \searrow 0}\mu(V_r(\alpha))=\mu(\partial\alpha)=0$. Then we can choose $\tau>0$ such that $\mu(V_r(\alpha))<\epsilon$ for any $0<r\le\tau$.

For any subset $F$ of $X$, let $1_{F}$ be its indicator function. For each $n\ge 1$ denote

$$A_n:=\{y\in X:\frac{1}{k}\sum_{i=0}^{k-1}1_{V_\tau(\alpha)}(T^{i}y) < 2\sqrt{\epsilon} ~\text{ for any }~ k\ge n\}.$$
Since $\mu(V_\tau(\alpha))<\epsilon$, there exists $\ell_0\ge 1$ such that
\begin{align} \label{eq:An measure lower}
\mu(A_n)> 1-2\sqrt{\epsilon},\quad \forall~n\ge \ell_0\,,
\end{align}
and here the ergodicity of $\mu$ is not required. To see this, note that by Birkhoff's ergodic theorem and Egorov's theorem, $\frac{1}{k}\sum_{i=0}^{k-1}1_{V_\tau(\alpha)}\circ T^{i}$ converges to $\mbbe(1_{V_\tau(\alpha)}|\msce_T)$ almost uniformly. On the other hand, by Chebyshev's inequality,
\[ \mu\big(\{ y\in X :  \mbbe(1_{V_\tau(\alpha)}|\msce_T)(y)\ge \sqrt{\epsilon} \}\big)< \sqrt{\epsilon}\,.\]
Then \eqref{eq:An measure lower} follows easily.

Define $Q_n:=\{x\in X:\mu_x(A_n)\ge1-2\epsilon^{\frac{1}{4}}\}$. Then
$$ \mu(Q_n^c)\cdot2\epsilon^{\frac{1}{4}}\le\int_X \mu_x (A_n^c)d\mu(x)=\mu(A_n^c)<2\sqrt{\epsilon}\,,\quad \forall\,n\ge \ell_0\,.$$
Thus $\mu(Q_n)>1-\epsilon^{\frac{1}{4}}$ for $n\ge \ell_0$.


Recall that we denote $c=H_\mu(\alpha|\mscf)$. By \eqref{eq:ergodic SMB local} in Corollary~\ref{cor:SMB local}  there exists $X_0\subset X$ with $\mu(X_0)=1$ such that for $x\in X_0$, $\Xi_x:=\xi(x)\cap X_0$ and $\alpha_n:=\bigvee\limits_{i=0}^{n-1}T^{-i}\alpha$, we have:
\begin{equation}\label{eq-a:local entropy apply}
\mu_x(\Xi_x)=1, \qtext{and} \lim_{n\rightarrow +\infty}\frac{-\log\mu_x \big( \alpha_n(y) \big)}{n}=c,\ \ \forall\,y\in \Xi_x\,.
\end{equation}
Let $I=X_0\cap Q_{\ell_0}$. Then $\mu(I)=\mu(Q_{\ell_0})>1-\epsilon^{\frac{1}{4}}$.

Now fix an arbitrary $x\in I$ and let
$$B_n:=\big\{ y\in \Xi_x: \frac{-\log\mu_{x}\big( \alpha_k(y) \big)}{n}\ge c-\epsilon \text{ for any } k\ge n \big\}\,.$$
Then by \eqref{eq-a:local entropy apply} we can find  $\ell_1\ge \ell_0$ such that $\mu_{x}(B_{\ell_1})\ge 1-\epsilon^{\frac{1}{4}}$.
Let $E:=A_{\ell_1}\cap B_{\ell_1}$. From \eqref{eq:An measure lower} and the definition of $E$ we know that $\mu_{x}(E)>1-3\epsilon^{\frac{1}{4}}$.

For nonempty $V\subset X$, let $\alpha(V)$ denote the unique atom in $\alpha$ that contains $V$ once it makes sense. For any $n\ge 1$ and any $V\in \alpha_n$, $\alpha(T^k V)$ is well defined for each $0\le k<n$. Let $\rho_n$ be a {\it pseudo-metric}\footnote{More precisely, $\rho_n(V,W)=0$ may not imply $V=W$, and except for this, $\rho_n$ satisfies all the other axioms of a metric.} on $\alpha_n$ defined by
\[\rho_n(V,W):=\frac{1}{n}\cdot\#\{0\le i< n:\alpha(T^i V)\ne \alpha(T^i W)\}. \]
Observe that if $z\in B_n(y,\tau,T)$, then for any $0\le i<n$, either $T^{i}y$ and $T^{i}z$ belong to the same element of $\alpha$ or $T^{i}y\in V_\tau(\alpha)$. Hence when $y\in E$ and $n\ge \ell_1$,
\[ z\in B_n(y,\tau,T)  \implies  \rho_n(\alpha_n(y),\alpha_n(z)) <2\sqrt{\epsilon}\,.\]
In other words, if we denote the $2\sqrt{\epsilon}$-``open ball" of $W$ under $\rho_n$ by $\mclv_n(W)$,  i.e.
\[ \mclv_n(W):= \big\{V\in \alpha_n \mid \rho_n(V,W) <2\sqrt{\epsilon} \,\big\} \,\]
for each $W\in \alpha_n$, then
\begin{equation}\label{eq:Bowen ball cover}
  B_n(y,\tau,T) \subset \bigcup_{V\in \mclv_n(\alpha_n(y))} V\,, \quad \forall\,y\in E,~n\ge \ell_1\,.
\end{equation}
On the other hand, by the definition of $\mclv_n(W)$, we have:
\[ \#\mclv_n(W) \le \sum_{i=0}^{m}{n \choose i}(\#\alpha-1)^{i} < m \cdot{n \choose  m} (\#\alpha-1)^{m}\,, \qtext{for}  m= \lceil 2n\sqrt{\epsilon}\rceil \,.\]
Then Stirling's formula  implies that (see for example \cite[Page 144]{Ka}) for $\epsilon>0$ small enough, there exists $\ell_2\ge 1$ such that the following holds:
\begin{equation}\label{eq:Striling upper}
  \#\mclv_n(W) \le\exp\big( (\Delta+\epsilon)n\big)\,,\quad\forall\, n\ge\ell_2 \,,
\end{equation}
where $\Delta$ is  defined in \eqref{eq:Stirling const}.

We shall choose $D$  appearing \eqref{eq:claim} as a subset of $E$ by removing some ``bad parts'' $E_n$ for large $n$ that will be specified below. To this end, let
\[\mclf_n:=\big\{V\in \alpha_n:\mu_{x}(V)>\exp \big((-c+2(\Delta+\epsilon))n \big) \big\}\,,\]
and let
\[E_n:= \bigcup_{V\in\mclf_n}\big\{y\in E : \rho_n(\alpha_n(y), V)<2\sqrt{\epsilon}\,\big\}  \,.\]
By the definition of $\mclf_n$ and noting that $\mu_{x}(X)=1$, we have:
\[
\#\mathcal{F}_n\le\exp\big(c-2(\Delta+\epsilon))n\big)\,,\quad \forall n\ge 1\,.
\]
To estimate the size of $E_n$, for each $n\ge 1$, let
\[
\mclg_n := \bigcup_{W\in\mclf_n} \big\{V\in \mclv_n(W): \mu_{x}(V)<\exp \big((-c+\epsilon)n \big)\big\}\subset \bigcup_{W\in\mclf_n}  \mclv_n(W) \,.
\]
It follows that
\[
\#\mathcal{G}_n\le \exp\big( (\Delta+\epsilon)n \big) \cdot\#\mathcal{F}_n\le \exp\big((c-(\Delta+\epsilon))n\big)\,,\quad\forall\, n\ge\ell_2 \,.
\]
By the definition of $E_n$ and noting that for any $y\in E$,  $\mu_x(\alpha_n(y))<\exp \big((-c+\epsilon)n \big)$ holds for $n\ge \ell_1$, we obtain that
\[
E_n \subset \bigcup_{W\in\mclg_n} W \implies \mu_{x}(E_n)\le\exp\{(-c+\epsilon)n\}\cdot\#\mathcal{G}_n\le\exp(-\Delta n)\,,\quad\forall\, n\ge\ell_3 \,,
\]
where $\ell_3=\max\{\ell_1,\ell_2\}$. Then there exists $\ell\ge \ell_3$ such that $\sum_{n=\ell}^\infty \mu_{x}(E_n)<\epsilon^{\frac{1}{4}}$. Let $D=E\setminus \cup_{n=\ell}^\infty E_n$. Then $\mu_{x}(D)>1-4\epsilon^{\frac{1}{4}}$.  Given $y\in D$ and $n\ge \ell$, since $y\in E\setminus E_n$, it is clear that for each $V\in \alpha_n$ with $\rho_n(V,\alpha_n(y))<2\sqrt{\epsilon}$, one has
$$\mu_{x}(V)\le\exp\{(-c+2(\Delta+\epsilon))n\}.$$
Moreover combing this with \eqref{eq:Bowen ball cover} and \eqref{eq:Striling upper} we have
\begin{align*}
\mu_{x}\big(B_n(y,\tau,T)\big)&\le \exp\{(\Delta+\epsilon)n\}\cdot\exp\{(-c+2(\Delta+\epsilon))n\}\\
&=\exp\{(-c+3(\Delta+\epsilon))n\}.
\end{align*}
Thus for any $y\in D$ and $n\ge \ell$,
$$\frac{-\log\mu_{x}\big( B_n(y,\tau,T) \big)}{n}\ge c-3(\Delta+\epsilon)\,$$
which completes the proof.
\end{proof}

\bibliographystyle{plain}             

\end{document}